\documentclass[11pt,a4paper]{amsart}
\usepackage{amsfonts,amsmath,amssymb,amscd}
\usepackage{fancybox}
\usepackage[latin1]{inputenc}
\usepackage[usenames]{color}
\usepackage{graphicx}
\usepackage{subfig}
\usepackage{float}
\usepackage{fullpage}
\usepackage{epsfig,amssymb}
\usepackage{mathtools}

\theoremstyle{plain}
\newtheorem{thm}{Theorem}[section]

\newtheorem{prop}[thm]{Proposition}

\theoremstyle{definition}

\newtheorem{rem}[thm]{Remark}

\def \r{\mbox{${\mathbb R}$}}

\def \s{\mbox{${\mathbb S}$}}

\DeclareMathOperator{\grad}{grad}
\DeclareMathOperator{\trace}{trace}

\DeclareMathOperator{\Ricci}{Ric}

\begin{document}
	
\title{Totally biharmonic hypersurfaces in space forms and 3-dimensional BCV spaces}

\author{S. Montaldo}
\address{Universit\`a degli Studi di Cagliari\\
Dipartimento di Matematica e Informatica\\
Via Ospedale 72\\
09124 Cagliari, Italia}
\email{montaldo@unica.it}

\author{A. P\'ampano}
\address{Department of Mathematics\\
University of the Basque Country\\
Aptdo. 644\\
48080, Bilbao, Spain.}
\email{alvaro.pampano@ehu.eus}
\date{\today}

\begin{abstract}
A hypersurface is said to be totally biharmonic if all its geodesics are biharmonic curves in the ambient space. We prove that a totally biharmonic hypersurface into a space form is an isoparametric biharmonic hypersurface, which allows us to give the full classification of totally biharmonic hypersurfaces in these spaces. \\
Moreover, restricting ourselves to the 3-dimensional case, we show that totally biharmonic surfaces into Bianchi-Cartan-Vranceanu spaces are isoparametric surfaces and we give their full classification. In particular, we show that, leaving aside surfaces in the 3-dimensional sphere, the only non-trivial example of a totally biharmonic surface appears in the product space $\mathbb{S}^2(\rho)\times \mathbb{R}$.\\
\end{abstract}

\thanks{Work partially supported by Fondazione di Sardegna (Project STAGE) and Regione Autonoma della Sardegna (Project KASBA). The second author has been partially supported by MINECO-FEDER grant PGC2018-098409-B-100, Gobierno Vasco grant IT1094-16 and Programa Posdoctoral del Gobierno Vasco, 2018. He also wants to thank the Department of Mathematics and Computer Science of the University of Cagliari for the warm hospitality during his stay.}
\maketitle

\section{Introduction}

A hypersurface $M^{n-1}$ into a $n$-dimensional Riemannian manifold $N^n$ is called \emph{totally biharmonic} if all the geodesics of $M^{n-1}$ are biharmonic curves of $N^n$.
The notion of totally biharmonic hypersurfaces was introduced in \cite{IM}, where the authors characterized these hypersurfaces in terms of their extrinsic geometry. More precisely, in \cite{IM}, it was proved that totally biharmonic hypersurfaces can be characterized as solutions of a system of equations involving the shape operator, its derivatives and the sectional curvature of the ambient space.  We recall this characterization in Proposition \ref{condition}.

For the purpose of this paper we are going to say that a hypersurface $M^{n-1}$  into $N^n$ is a \emph{biharmonic hypersurface} if the following system is verified
\begin{equation*}
\begin{dcases}
\Delta H+H \lvert S_\eta \rvert^2-H \Ricci (\eta,\eta)=0\\
2S_\eta(\grad H)+ (n-1) H \grad H-2 H \Ricci (\eta)^T=0\,,
\end{dcases}
\end{equation*}
where $S_\eta$ is the shape operator, $H$ is the mean curvature function and $\Ricci(\eta)^T$ is the tangent component of the Ricci curvature of $N^n$ in the direction of the vector field $\eta$. 

To state the first result of the paper we recall that a hypersurface into a space form $N^n=N^n(\rho)$, of constant sectional curvature $\rho$, is \emph{isoparametric} if the principal curvatures are constant (\cite{Ca1}). We then have the following characterization. 

\begin{thm}\label{TB=IB} Let $M^{n-1}$  be a totally biharmonic hypersurface into a space form $N^n=N^n(\rho)$ with constant sectional curvature $\rho$. Then $M^{n-1}$ is an isoparametric biharmonic hypersurface.
\end{thm}

This result gives us a way to classify totally biharmonic hypersurfaces in space forms. In fact, the classification of biharmonic  isoparametric hypersurfaces in space forms is well-known (see, for instance, \cite{urakawa}) and, therefore, by a direct check one can obtain the corresponding local classification of totally biharmonic hypersurfaces (see Theorem~\ref{clas}).

However, in general, it may be interesting to obtain the classifications directly from the definition of totally biharmonic hypersurfaces, since this may give a new insight in the theory of biharmonic submanifolds. 

To this aim we concentrate ourselves to the classification of totally biharmonic surfaces in the Bianchi-Cartan-Vranceanu (BCV) spaces $N(a,b)$. These spaces, as we shall detail in \S3, represent a local model of the 3-dimensional homogeneous manifolds with isometry group of dimension 4 or 6 except for the hyperbolic space. 
We first prove that 
 a totally biharmonic surface into a BCV space $N(a,b)$, $4\,a\neq b^2$,  is either totally geodesic or a rotational Hopf cylinder (Theorem~\ref{hopf}). As a consequence, we conclude with the following local classification of totally biharmonic surfaces into BCV spaces.

\begin{thm}\label{relconj} Let $M$ be a totally biharmonic surface into a BCV space $N(a,b)$. Then $M$ is either a part of a totally geodesic surface or one of the following:
\begin{itemize}
\item[(i)] a part of the Clifford torus $\mathbb{S}^1(\sqrt{2\,\rho\,}\,)\times\mathbb{S}^1(\sqrt{2\,\rho\,}\,)\subset\mathbb{S}^3(\rho)$;
\item[(ii)] a part of the totally umbilical sphere $\mathbb{S}^2(2\,\rho)\subset\mathbb{S}^3(\rho)$;
\item[(iii)] a part of the cylinder $\mathbb{S}^1(\sqrt{2\,\rho\,}\,)\times\mathbb{R}\subset\mathbb{S}^2\left(\rho\right)\times\mathbb{R}$.
\end{itemize}
\end{thm}


\section{Totally Biharmonic Hypersurfaces}

\emph{Harmonic maps} $\varphi\,: (M,g)\rightarrow (N,h)$ between Riemannian manifolds are the critical points of the energy functional
\begin{equation}
E(\varphi)=\frac{1}{2}\int_M \lvert d\varphi \rvert ^2 v_g\,  \nonumber
\end{equation}
and the corresponding Euler-Lagrange equation is given by the vanishing of the tension field
$\tau(\varphi)=\trace \nabla d\varphi$ (see \cite{Eells-Sampson}).
In \cite{ES}, J. Eells and J. H. Sampson suggested to study \emph{biharmonic maps}, which are the critical points of the \emph{bienergy} functional
\begin{equation}
E_2(\varphi)=\frac{1}{2}\int_M \lvert \tau(\varphi)\rvert ^2 v_g\, . \nonumber
\end{equation} 
The first variation formula of the bienergy was derived by G. Y. Jiang in \cite{J}, proving that the Euler-Lagrange equation for $E_2$ is 
\begin{equation}
\tau_2(\varphi)=-\Delta \tau(\varphi)-\trace R^N(d\varphi,\tau(\varphi))d\varphi=0\, , \label{bitensionfield}
\end{equation}
where $R^N$ is the curvature operator of $(N,h)$ defined by
\begin{equation}
R^N(X,Y)=\nabla_X\nabla_Y-\nabla_Y\nabla_X-\nabla_{[X,Y]}\,, \quad X,Y\in C(TN), \nonumber
\end{equation}
and $\Delta$ represents the rough Laplacian on $C(\varphi^{-1}(TN))$, which, for a local orthonormal frame $\{e_i\}_{i=1}^m$ on $M$, is defined by
\begin{equation}
\Delta=-\sum_{i=1}^m\left(\nabla_{e_i}^\varphi\nabla_{e_i}^\varphi-\nabla^\varphi_{\nabla_{e_i}^M e_i}\right)\, . \nonumber
\end{equation}
The equation $\tau_2(\varphi)=0$ is called the \emph{biharmonic equation}.

Let $\varphi: M^m\rightarrow N^n$ be an isometric immersion of a Riemannian manifold of dimension $m$ into another Riemannian manifold of dimension $n>m$. The positive integer $n-m$ is called the \emph{codimension} of the immersion $\varphi$. In particular, if $n-m=1$, we will say that $M^{n-1}$ is an \emph{hypersurface} of $N^n$.

Now, for an isometric immersion of a hypersurface $M^{n-1}$ into a Riemannian manifold $N^n$, let's denote by $\eta$ the unit normal vector field to $M^{n-1}$ in $N^n$ and by  $S_\eta$ the corresponding shape operator.


For isometric immersions of codimension one the bitension field can be nicely decomposed in its normal and tangential components as it is shown in the following theorem (see, for instance, \cite{BMO13,C84,LM08,O10}).

\begin{thm}\label{decompbitension} Let $\varphi:M^{n-1}\rightarrow N^n$ be an isometric immersion with unit normal vector field $\eta$ and mean curvature function $H$. Then, the normal and tangential components of $\tau_2(\varphi)$ are respectively
\begin{equation}\label{eq:def-bih}
\begin{cases}
\Delta H+H \lvert S_\eta \rvert^2-H \Ricci (\eta,\eta)=0\\
2S_\eta(\grad H)+ (n-1) H \grad H-2 H \Ricci (\eta)^T=0
\end{cases}
\end{equation}
where $S_\eta$ is the shape operator and $\Ricci(\eta)^T$ is the tangent component of the Ricci curvature of $N$ in the direction of the vector field $\eta$.
\end{thm}

We shall say that an isometric immersion $\varphi:M^{n-1}\rightarrow N^n$ is {\it biharmonic} if the corresponding bitension field vanishes, that is if system~\eqref{eq:def-bih} of Theorem~\ref{decompbitension} is satisfied. Moreover, as introduced in \cite{LM08}, we shall say that a hypersurface $M^{n-1}$ is \emph{biminimal} if only the normal component of $\tau_2(\varphi)$ vanishes, that is, according to Theorem~\ref{decompbitension}, if the following equation is verified
\begin{equation}
\Delta H+H \lvert S_\eta \rvert^2-H \Ricci(\eta,\eta)=0\, . \label{a-b}
\end{equation}
Observe that minimal hypersurfaces ($H=0$) are clearly biharmonic and that biharmonic hypersurfaces are obviously biminimal. Moreover, any constant mean curvature biminimal hypersurface is biharmonic provided that the tangent component of the Ricci curvature of $N^n$, in the direction of the vector field $\eta$, vanishes, that is $\Ricci (\eta)^T=0$.

Let now $\gamma: I\rightarrow N^n$ be a curve parametrized by arc-length, from an open interval $I\subset\mathbb{R}$ to a Riemannian manifold $N^n$. In this case, putting $\gamma' =\mathbf{t}$, the biharmonic equation \eqref{bitensionfield} reduces to
\begin{equation}
\tau_2(\gamma)=\nabla_{\mathbf{t}}^{3} \,\mathbf{t}+R^N\left(\nabla_\mathbf{t}\, \mathbf{t},  \mathbf{t}\right) \mathbf{t}=0\, . \label{bcurve}
\end{equation}
Solutions of this equation are called \emph{biharmonic curves}. In particular, any geodesic ($\nabla_\mathbf{t}\, \mathbf{t}=0$) is a biharmonic curve. We say that a biharmonic curve is \emph{proper} if it is not a geodesic. 

A natural question is to understand when the geodesics of a given hypersurface $M^{n-1}$ of a Riemannian manifold $N^n$ are always biharmonic curves in the ambient space. More precisely, consider a codimension one immersion $\varphi: M^{n-1}\rightarrow N^n$, we say that the immersed hypersurface  $M^{n-1}$ is a \emph{totally biharmonic hypersurface} if all geodesics of $M^{n-1}$ are biharmonic curves in $N^n$. This notion was first introduced in \cite{IM}, where the classification of totally biharmonic surfaces into 3-dimensional space forms was given. 

Notice that totally biharmonic hypersurfaces generalize the notion of totally geodesic hypersurfaces in the same way biharmonic curves generalize geodesics.

In order to characterize the totally biharmonic hypersurfaces, we shall recall  the following proposition first proved in \cite{IM}.

\begin{prop} \label{condition} Let $\varphi: M^{n-1}\rightarrow N^n$ be a codimension one isometric immersion. Then, the immersed hypersurface is totally biharmonic if and only if, for any point $p\in M^{n-1}$, one of the followings is satisfied,
\begin{itemize}
\item $S_\eta \equiv 0$ and $\varphi$ is totally geodesic; or,
\item $S_\eta$ is a solution of the following system,
\begin{eqnarray}
&& \langle S_\eta(X),X\rangle =cst\neq 0\, , \label{s1}\\
&& \langle S_\eta(X), S_\eta(X)\rangle=K^N(X,\eta)\, ,\label{s2}\\
&& \langle \nabla_X S_\eta(X),Y\rangle=-\langle R^N(X,\eta)X,Y\rangle\, , \label{s3}
\end{eqnarray}
where $X$, $Y$ are any orthonormal vectors of $T_p M^{n-1}$.
\end{itemize}
\end{prop}

\section{Totally Biharmonic Hypersurfaces of Space Forms}

In this section we are going to classify totally biharmonic hypersurfaces of $n$-dimensional space forms. Let us denote by $N^n=N^n(\rho)$ any Riemannian space form of dimension $n$, where $\rho$ represents the constant sectional curvature. If $\rho<0$, we have the $n$-dimensional \emph{hyperbolic space}, simply denoted by $\mathbb{H}^n(\rho)$; when $\rho=0$, $N^n(\rho)=\mathbb{R}^n$ denotes the \emph{Euclidean space} of dimension $n$; and, finally, for positive $\rho$ we recover the $n$-dimensional \emph{round sphere}, $\mathbb{S}^n(\rho)$. When we want to explicitly indicate the radius $R=1/\sqrt{\rho}$ of the sphere $\mathbb{S}^n(\rho)$, we will use the notation $\mathbb{S}^n\left[R\right]$.

We are now in the right position to prove Theorem \ref{TB=IB}.

\subsection{Proof of Theorem \ref{TB=IB}}

Let us assume that $M^{n-1}$ is a totally biharmonic hypersurface of $N^n(\rho)$. Then, by Proposition \ref{condition}, we have that $M^{n-1}$ is either totally geodesic or its shape operator verifies equations \eqref{s1}--\eqref{s3}.
First, if $M^{n-1}$ is totally geodesic, then it is clearly isoparametric. Moreover, in this case, $M^{n-1}$ is also minimal and, therefore, biharmonic, proving the statement. \\
On the other hand, suppose that $M^{n-1}$ is not totally geodesic. Then, its shape operator $S_\eta$ must verify \eqref{s1}--\eqref{s3}. Now, for a fixed point $p\in M^{n-1}$ and any principal direction, $e_i$, $i=1,2,...,n-1$, of the hypersurface, we have that $S_\eta\left(e_i\right)=\lambda_i\,e_i$. Then, from equation \eqref{s2} we have that the associated principal curvature $\lambda_i$ must verify $\lambda_i^2=\rho$. Furthermore, using equation \eqref{s1}, $\lambda_i$ is a non-vanishing constant. In particular, this is only possible when $\rho>0$ and $N^n(\rho)=\mathbb{S}^n(\rho)$. In any case, the hypersurface $M^{n-1}$ is isoparametric. Moreover, since all principal curvatures are constant, we also have that $M^{n-1}$ has constant mean curvature, $H=H_o$. \\
If $H_o=0$, then the hypersurface is minimal and, as above, it is also biharmonic and we are done. Therefore, we consider in what follows that $H_o\neq 0$. In this case, by substitution in \eqref{a-b}, we conclude that $M^{n-1}$ is biminimal if and only if the following equation is verified
\begin{equation}\label{eq-biminimal-teo}
\lvert S_\eta\rvert^2=\Ricci \left(\eta,\eta\right).
\end{equation}
At this step, evaluating equation \eqref{s2} at any principal direction $e_i$, we obtain that
\begin{equation}
\lambda_i^2=\langle S_\eta(e_i),S_\eta(e_i)\rangle =K^N\left(e_i,\eta\right).\nonumber
\end{equation}
Therefore, using the above identity, we obtain
\begin{equation}
\lvert S_\eta\rvert^2=\sum_{i,j=1}^{n-1}\langle S_\eta\left(e_i\right),e_j\rangle^2=\sum_{i,j=1}^{n-1}\lambda_i^2\langle e_i,e_j\rangle^2=\sum_{i=1}^{n-1}\lambda_i^2=\sum_{i=1}^{n-1}K^N\left(e_i,\eta\right)=\Ricci \left(\eta,\eta\right),\nonumber
\end{equation}
proving that equation \eqref{eq-biminimal-teo} is satisfied and, therefore, $M^{n-1}$ is biminimal. \\
Finally, as mentioned in \S 2, constant mean curvature biminimal hypersurfaces are biharmonic if the tangent component of the Ricci curvature of $N^n$ in the direction of the vector field $\eta$ vanishes, that is $\Ricci (\eta)^T=0$. In particular, if $N^n$ is a space with constant sectional curvature, then, for any hypersurface into $N^n=N^n(\rho)$, $\Ricci (\eta)^T=0$. We thus have that $M^{n-1}$ is an isoparametric biharmonic hypersurface. \hfill $\square$
\\

To end up this section, we prove the following classification.

\begin{thm}\label{clas} Let $M^{n-1}$ be a hypersurface into a space form $N^n(\rho)$. If $\rho\leq 0$, then $M^{n-1}$ is locally a totally geodesic hypersurface. On the other hand, if $\rho>0$, then up to rescaling, $N^n(\rho)\cong\mathbb{S}^n(1)=\mathbb{S}^n[1]$ and $M^{n-1}$ is, locally, either a totally geodesic hypersurface or one of the followings:
\begin{itemize}
\item[(i)] a part of the totally umbilical sphere $\mathbb{S}^{n-1}[1/\sqrt{2}\,]$;
\item[(ii)] a part of the generalized Clifford torus $\mathbb{S}^p[1/\sqrt{2}\,]\times\mathbb{S}^{n-p-1}[1/\sqrt{2}\,]$  for any $p>0$.
\end{itemize}
\end{thm}
\begin{proof}  
From Theorem \ref{TB=IB} we know that a totally biharmonic hypersurface into $N^n(\rho)$ must be isoparametric and biharmonic. Therefore, if $M^{n-1}$ is not totally geodesic we just need to check when conditions \eqref{s1}-\eqref{s3} are verified for an isoparametric biharmonic hypersurface. Let us take any principal direction $e_i$, $i=1,...,n-1$, of $M^{n-1}$ and denote by $\lambda_i$ its associated principal curvature. Now, since $N^n(\rho)$ has constant sectional curvature $\rho$, equation \eqref{s2} reads
$$\lambda_i^2=\rho$$
for any $i=1,...,n-1$, which is clearly a contradiction if $\rho\leq 0$. That is, in these cases, the only totally biharmonic hypersurfaces are the totally geodesic ones.\\
On the other hand, if $\rho>0$, up to rescaling we can assume that $N^n(\rho)\cong\mathbb{S}^n(1)=\mathbb{S}^n[1]$. In this case, we have $\lambda_i^2=1$, $i=1,...,n-1$. Thus, the hypersurface has at most two distinct principal curvatures. Proper (non minimal) biharmonic hypersurfaces with at most two distinct principal curvatures were classified in \cite{new}, where the authors proved that they are either a part of $\mathbb{S}^{n-1}[1/\sqrt{2}\,]$ or a part of $\mathbb{S}^p[1/\sqrt{2}\,]\times\mathbb{S}^{n-p-1}[1/\sqrt{2}\,]$ with $p\neq n-p-1$. We now show that the non minimal biharmonic hypersurfaces described above are totally biharmonic. We begin with the case of $\mathbb{S}^{n-1}[1/\sqrt{2}\,]$. Notice that equations \eqref{s1}-\eqref{s3} are verified in this case. Indeed, for this hypersurface we have that $\lambda_i=1$, which is enough to prove \eqref{s1} and \eqref{s2}. Finally, the left-hand side of equation \eqref{s3}, in this case, becomes
\begin{equation}\label{eq:new-proof}
\langle \nabla_{e_i} S_\eta(e_i), e_j\rangle= \langle\nabla_{e_i} e_i, e_j\rangle=0,
\end{equation}
while the right-hand side also vanishes since $R^{\small \s^n}(e_i,\eta)e_i$ is normal to the hypersurface. Observe that the second equality in \eqref{eq:new-proof} is true since, in the case of the sphere, geodesics are lines of curvature (see also the computations in \cite{IM}).\\
Now, we need to consider the case of $\mathbb{S}^p[1/\sqrt{2}\,]\times\mathbb{S}^{n-p-1}[1/\sqrt{2}\,]$ with $p\neq n-p-1$. For these generalized Clifford tori, by a direct computation, we show that all geodesics are totally biharmonic curves in $\mathbb{S}^n[1]$. If the geodesic lies fully in one of the factors we are done, since it will also be a geodesic of $\mathbb{S}^{n-1}[1/\sqrt{2}\,]$, which we have just proved to be totally biharmonic. Hence, consider any geodesic, $\gamma$, of $\mathbb{S}^p[1/\sqrt{2}\,]\times\mathbb{S}^{n-p-1}[1/\sqrt{2}\,]\subset \r^{p+1}\times\r^{n-p}$ that does not lie fully in any of the factors. Then, there exist geodesics $\alpha$ and $\beta$ in each factor such that $\gamma(s)=\left(\alpha(a s),\beta(b s)\right)$, where $a$, $b$ are suitable constants satisfying $a^2+b^2=1$, and $s$ represents the arc-length parameter of $\gamma$. Since $\alpha$ and $\beta$ are geodesics in some sphere of radius $1/\sqrt{2}$, i.e. great circles, they can be parametrized, respectively, by
\begin{eqnarray*}
\alpha\left(as\right)&=&\cos\left(\sqrt{2} a s\right) v_1+\sin\left(\sqrt{2} a s\right) v_2\,,\\
\beta\left(bs\right)&=&\cos\left(\sqrt{2} b s\right) v_3+\sin\left(\sqrt{2} b s \right) v_4\,,
\end{eqnarray*} 
where $v_{i}$, $i=1,2,3,4$ are orthogonal constant vectors ($v_1,v_2\in \r^{p+1}$, while $v_3,v_4\in\r^{n-p}$) satisfying $\lVert v_{i}\rVert^2=1/2$. Therefore, the geodesic $\gamma$ is given by
$$\gamma(s)=\cos\left(\sqrt{2} a s\right) v_1+\sin\left(\sqrt{2} a s\right) v_2+\cos\left(\sqrt{2} b s\right) v_3+\sin\left(\sqrt{2} b s \right) v_4\,,$$
which belongs to the family of biharmonic curves in $\mathbb{S}^n[1]$, described in Proposition 4.4 of \cite{nuevo}. This proves that  $\mathbb{S}^p[1/\sqrt{2}\,]\times\mathbb{S}^{n-p-1}[1/\sqrt{2}\,]$ with $p\neq n-p-1$ is totally biharmonic in $\mathbb{S}^n[1]$.\\
Moreover, if the hypersurface is minimal, isoparametric and with at most two distinct principal curvatures, then it is either totally geodesic or the generalized Clifford torus $\mathbb{S}^p[1/\sqrt{2}\,]\times\mathbb{S}^{p}[1/\sqrt{2}\,]$ with $2p=n-1$. A similar argument as above proves that these generalized Clifford tori are totally biharmonic. This finishes the proof. 
\end{proof}

\begin{rem}
As a consequence of the results in this section we can conclude that a hypersurface $M^{n-1}$ into a space form $N^n(\rho)$ is totally biharmonic if and only if it is isoparametric and biharmonic with at most two distinct principal curvatures.
\end{rem}

\section{Totally biharmonic surfaces in 3-dimensional homogeneous spaces}

A Riemannian manifold $(N^n, h)$ is said to be \emph{homogeneous} if for every two points $p$ and $q$ in $N^n$, there exists an isometry of $N^n$ mapping $p$ into $q$. For homogeneous 3-dimensional manifolds ($n=3$) there are three possibilities for the degree of rigidity, since they may have an isometry group of dimension $6$, $4$ or $3$. 
The maximum rigidity, $6$, corresponds to the space forms, $N^3(\rho)$.
%
The homogeneous 3-dimensional spaces with isometry group of dimension four include, amongst its simply connected members, the product spaces $\mathbb{S}^2(\rho)\times\mathbb{R}$ and $\mathbb{H}^2(\rho)\times\mathbb{R}$; the Berger spheres; the Heisenberg group; and, the universal covering of the special linear group $Sl(2,\mathbb{R})$. Cartan in \cite{Ca} showed that all homogeneous 3-manifolds with group of isometries of dimension $4$ can be described by a Bianchi-Cartan-Vranceanu space $N(a,b)$, where $4\,a\neq b^2$, see Subsection \ref{BCV}. Finally, if the dimension of the isometry group is $3$, the homogeneous 3-space is isometric to a general simply connected Lie group with left invariant metric.


In this section we will focus only on totally biharmonic surfaces into the homogeneous 3-manifolds with group of isometries of dimension $4$.

\subsection{Bianchi-Cartan-Vranceanu Spaces}\label{BCV}

BCV spaces (see \cite{Ca,Vr}) are described by  the following two-parameter family of Riemannian metrics
\begin{equation}\label{1.1}
h_{a,b} =\frac{dx^{2} + dy^{2}}{[1 + a(x^{2} + y^{2})]^{2}} +  \left(dz + \frac{b}{2} \frac{ydx - xdy}{[1 + a(x^{2} + y^{2})]}\right)^{2},\quad a,b \in {\mathbb{R}}
\end{equation}
defined on $N^3=\{(x,y,z)\in\mathbb{R}^3\,;\, \lambda_a=1+a\left(x^2+y^2\right)>0\}$. We are going to denote these BCV spaces by $N(a,b)$. We recall that these spaces are the only simply connected homogeneous 3-manifolds
admitting the structure of a Killing submersion, \cite{M}. 

Now, if we consider the orthonormal basis of vector fields given by $\{E_1,E_2,E_3\}$ where
\begin{equation}
E_1=\lambda_a\, \frac{\partial}{\partial x}-\frac{b\, y}{2}\frac{\partial}{\partial z}\, , \quad E_2=\lambda_a\,\frac{\partial}{\partial y}+\frac{b\, x}{2}\frac{\partial}{\partial z}\, ,\quad E_3=\frac{\partial}{\partial z}\, ,\label{orthonormalframe}
\end{equation}
we can write the expressions for the Levi-Civita connection as
\begin{equation}\label{eq:levi-civita-BCV}
\begin{array}{lll}
\nabla_{E_1} E_1=2\,a\, y\, E_2\, , \quad & \nabla_{E_1} E_2=-2\, a\, y\, E_1+\frac{b}{2}E_3\, ,\quad & \nabla_{E_1}E_3=-\frac{b}{2}E_2 \, ,  \\
\nabla_{E_2} E_1=-2\,a\,y\, E_1+\frac{b}{2}E_3\, ,\quad & \nabla_{E_2}E_2=2\,a\, x\, E_1\, ,\quad & \nabla_{E_2}E_3=\frac{b}{2}E_1\, ,  \\
\nabla_{E_3}E_1=-\frac{b}{2}E_2\, , \quad & \nabla_{E_3}E_2= \frac{b}{2}E_1\, ,\quad & \nabla_{E_3}E_3 =0\, . 
\end{array}
\end{equation}
Moreover, the nonzero components of the curvature tensor can be computed, obtaining
\begin{equation}
R_{1212}=4\,a-\frac{3}{4} b^2\, , \quad\quad R_{1313}=R_{2323}= \frac{b^2}{4}\, . \label{R}
\end{equation}

Observe that, from the above expressions of curvature tensor, if $4\,a=b^2$ then $N(a,b)$ represents a 3-dimensional space form. Therefore, from now on, we are going to assume that $4\,a\neq b^2$. In these cases, as mentioned before, the family of metrics \eqref{1.1} includes all three-dimensional homogeneous metrics whose group of isometries has dimension $4$. The classification of these spaces is as follows (see, for instance \cite{M})
\begin{itemize}
\item If $a=0$ and $b\neq 0$, we have that $N(a,b)\cong \mathbb{H}_3$, the \emph{Heisenberg group}.
\item If $a>0$ and $b=0$, $N(a,b)\cong\left(\mathbb{S}^2(4\,a) -\{\infty\}\right) \times \mathbb{R}$.
\item If $a<0$ and $b=0$, $N(a,b)\cong \mathbb{H}^2(4\,a)\times \mathbb{R}$.
\item If $a>0$, $b\neq 0$ and $4\,a\neq b^2$, then $N(a,b)\cong SU(2)- \{\infty\}$.
\item And, finally, if $a<0$ and $b\neq 0$, we have that $N(a,b)\cong \widetilde{Sl}(2,\mathbb{R})$.
\end{itemize}

Moreover, the Lie algebra of the infinitesimal isometries of $N(a,b)$ with $4\,a\neq b^2$ admits the following basis of Killing vector fields
\begin{equation}\label{X4}
\begin{array}{lll}
X_1&=&\left(1-\frac{2\,a\, y^2}{\lambda_a}\right) E_1+\frac{2axy}{\lambda_a} E_2+ \frac{b y}{\lambda_a} E_3\, , \\
X_2&=&\frac{2axy}{\lambda_a} E_1+\left(1-\frac{2ax^2}{\lambda_a}\right) E_2-\frac{bx}{\lambda_a} E_3\, ,\\
X_3&=&-\frac{y}{\lambda_a} E_1+\frac{x}{\lambda_a}E_2-\frac{b\left(x^2+y^2\right)}{2\lambda_a} E_3\, ,\\
X_4&=&E_3\, , 
\end{array}
\end{equation}
where $\{E_i\}$, $i=1,2,3$, is the orthonormal basis introduced in \eqref{orthonormalframe}. 

Then, a surface which stays invariant under the action of any Killing vector field, $\xi$, is called an \emph{invariant surface}. In particular, it can be checked that, the group of isometries of BCV spaces $N(a,b)$ with $4\,a\neq b^2$ contains a one-parameter subgroup isomorphic to $SO(2)$. Surfaces invariant by the action of $SO(2)$ are called {\em rotational surfaces}.

On the other hand, invariant surfaces under the action of the Killing vector field $X_4$, \eqref{X4}, are usually called \emph{Hopf cylinders}. These cylinders can be parametrized as $\mathbf{x}(s,t)=\psi_t(\widetilde{\alpha}(s))$, where $\widetilde{\alpha}(s)$ denotes an arc-length parametrized curve orthogonal to $X_4$ in $N(a,b)$ while $\{\psi_t\,;\, t\in\mathbb{R}\}$ is the one-parameter group of isometries associated to $X_4$. 

We recall that isoparametric surfaces of BCV spaces $N(a,b)$ with $4\,a\neq b^2$ have been classified in \cite{DVM}, where the authors proved that they must be either Hopf cylinders, horizontal slices or parabolic helicoids. 

\subsection{Biharmonic Curves in 3-Spaces}

In order to prove the local classification of totally biharmonic surfaces of BCV spaces $N(a,b)$ with $4\,a\neq b^2$, we are going to argue using directly the definition introduced in \cite{IM}. Therefore, we need to recall some facts of biharmonic curves in 3-spaces. 

Let's denote by $\gamma(s)$ an arc-length parametrized curve immersed in any 3-space $N^3$ and let's put $\gamma'(s)=\mathbf{t}(s)$. If the covariant derivative of the tangent vector field $\mathbf{t}(s)$ along $\gamma$ vanishes, that is $\nabla_\mathbf{t} \,\mathbf{t}(s)=0$, then $\gamma(s)$ is a geodesic curve. 

On the other hand, if $\gamma(s)$ is a unit speed non-geodesic smooth curve into $N^3$, 
then $\gamma(s)$ is a \textit{Frenet curve} of rank $2$ or $3$  and the standard \textit{Frenet frame} along $\gamma(s)$ is given by $\{\mathbf{t},\mathbf{n}, \mathbf{b}\}(s)$, where $\mathbf{n}$ and $\mathbf{b}$ are the \emph{unit normal} and \emph{unit binormal} to the curve, respectively, and $\mathbf{b}$ is chosen so that $\{\mathbf{t},\mathbf{n},\mathbf{b}\}$ is a positive orientated frame. Then the \emph{Frenet equations}
\begin{equation}\label{frenet1}
\left\{\begin{array}{lcl}
\nabla_\mathbf{t}\, \mathbf{t}(s)&=&\kappa(s) \mathbf{n} (s)\, , \\
\nabla_\mathbf{t} \mathbf{n}(s)&=&-\kappa(s) \mathbf{t}(s) +\tau(s) \mathbf{b}(s)\, ,\\
\nabla_\mathbf{t} \mathbf{b}(s)&=& -\tau (s) \mathbf{n} (s)\, ,
\end{array}
\right.
\end{equation}
define the \textit{curvature}, $\kappa(s)$ ($\kappa(s)>0$ if the rank is $3$), and the \textit{torsion}, $\tau(s)$, along $\gamma(s)$ (do not confuse the notation with the tension field $\tau(\varphi)$ defined in \S 2).

Now, substituting the Frenet equations \eqref{frenet1} in \eqref{bcurve}, the system of equations describing proper biharmonic curves can be written in 3-spaces as
\begin{equation}\label{b1}
\left\{
\begin{array}{rcl}
 \kappa&=&\kappa_o\neq 0\,, \\
 \kappa^2+\tau^2&=&K^N(\mathbf{t},\mathbf{n})\, ,\\
 \tau' &=&-\langle R^N(\mathbf{t},\mathbf{n})\mathbf{t},\mathbf{b}\rangle\, ,
\end{array}
\right.
\end{equation}
which represent the tangent, normal and binormal components of $\tau_2(\gamma)$, \eqref{bcurve}, respectively.

Moreover, if the 3-space is a BCV space $N(a,b)$, with $4\,a\neq b^2$, then, using the curvature tensors \eqref{R},  system \eqref{b1} can be rewritten as (see \cite[Section 5.1]{CMOP}  for details)
\begin{eqnarray}
&& \kappa=\kappa_o\neq 0 \, , \label{b11}\\
&& \tau=\tau_o \, , \label{b12}\\
&& \mathbf{n}_3=0 \, , \label{b13}\\
&& \kappa^2+\tau^2=\frac{b^2}{4}-\left(b^2-4\,a\right)\mathbf{b}_3^2\, , \label{b14}
\end{eqnarray}
where $\mathbf{n}_3=\langle \mathbf{n}, E_3\rangle$ and $\mathbf{b}_3=\langle \mathbf{b}, E_3\rangle$.

\subsection{Classification of totally biharmonic surfaces into BCV Spaces}

By using the equations of proper biharmonic curves in BCV spaces $N(a,b)$, $4\,a\neq b^2$, given by \eqref{b11}-\eqref{b14}, and using the definition of totally biharmonic surfaces, we have the following first characterization.

\begin{thm}\label{hopf} Let $M^2$ be a non-totally geodesic surface into a Bianchi-Cartan-Vranceanu space $N(a,b)$ with $4\,a\neq b^2$. If $M^2$ is a totally biharmonic surface, then $a>0$ and $M^2$ is a rotational Hopf cylinder.
\end{thm} 
\begin{proof} Fix any point $p$ in a surface $M^2$ of a BCV space $N(a,b)$ with $4\,a\neq b^2$ and denote by $\eta$ the unit normal vector field to $M^2$. Since $M^2$ is not totally geodesic, there exists a geodesic $\gamma$ passing through $p$ which is not a geodesic of the ambient space $N(a,b)$. \\ 
Consider $\{\mathbf{t},\mathbf{n},\mathbf{b}\}$ the associated Frenet frame of $\gamma$ in $N(a,b)$. Now, since $\gamma$ is a geodesic of the surface $\mathbf{n}$ must be parallel to the unit normal to the surface. Thus, up to orientation, we can assume, locally, that $\mathbf{n}=\eta$.\\
Moreover, since $M^2$ is totally biharmonic, $\gamma$ must be proper biharmonic as a curve in the ambient space $N(a,b)$. Thus, equation \eqref{b13} must be verified, that is, we have that $$\mathbf{n}_3=\langle \mathbf{n}, E_3\rangle=\langle \eta, E_3\rangle=0.$$ Therefore, at the point $p$, $E_3$ is tangent to the surface.\\
Next, arguing in the same way for any point $p\in M^2$, we conclude that $E_3$ is always tangent to $M^2$.\\
We now consider the following adapted orthonormal frame $\{\hat{E}_1,\hat{E}_2,\hat{E}_3=\eta\}$ given by
\begin{equation}\label{eq:ecap}
\begin{cases}
\hat{E}_1=\cos\vartheta\, E_1+\sin\vartheta \, E_2\,,\\
\hat{E}_2=E_3\,, \\
\hat{E}_3=\sin\vartheta\,E_1-\cos\vartheta \, E_2\, , 
\end{cases}
\end{equation}
for some function $\vartheta(p)$. After a long straightforward computation, taking into account the formulae of the Levi-Civita connection \eqref{eq:levi-civita-BCV},  we can compute, with respect to this adapted orthonormal frame, the second fundamental form of $M^2$, $II$. In particular, from the symmetry of $II$, we get the following relation
\begin{equation}
II(\hat{E}_1,\hat{E}_2)=-\langle \nabla_{\hat{E}_1}\hat{E}_3, \hat{E}_2\rangle=\frac{b}{2}=II(\hat{E}_2,\hat{E}_1)=-\langle \nabla_{\hat{E}_2}\hat{E}_3, \hat{E}_1\rangle=\frac{b}{2}-\hat{E}_2(\vartheta)\, ,\nonumber
\end{equation}
which implies that $\vartheta$ is a constant function along the fibers of the Killing submersion, $\pi$, associated to $\hat{E}_2=E_3$. Therefore, after a reparametrization if needed, $M^2$ can be seen as a Hopf cylinder. We assume that $M^2$ is locally parametrized as $\mathbf{x}(s,t)=\psi_t(\widetilde{\alpha}(s))$, where $\{\psi_t\,;\,t\in\mathbb{R}\}$ denotes the one-parameter group of isometries associated to $E_3$ and $\widetilde{\alpha}(s)$ is an arc-length parametrized curve in $N(a,b)$ everywhere orthogonal to $E_3$.\\
Now, using formula (16) of \cite{GM}, we have that the extrinsic Gaussian curvature of $M^2$, $K_e$, verifies
\begin{equation}\label{eq:Ke}
 K_e=\det \left(S_\eta\right)=\lambda_1\lambda_2=\frac{b^2}{4}\,,
 \end{equation}
where $\lambda_i$, $i=1,2$, are the principal curvatures of $M^2$. Then, using equation \eqref{s1}, we see that for a principal direction $e_i$, $i=1,2$, the associated principal curvature, $\lambda_i$, is constant in the direction of $e_i$ (see also \cite{IM}). Moreover, differentiating \eqref{eq:Ke}  in the direction of $e_1$ we have that
$$e_1\left(K_e\right)=\lambda_1 e_1\left(\lambda_2\right)=e_1\left(\frac{b^2}{4}\right)=0\,.$$
Thus, since $\lambda_1\neq 0$ from equation \eqref{s1}, we conclude that $\lambda_2$ is also constant in the direction of $e_1$. Repeating the same argument for $e_2$ we get that the principal curvatures are constant in a neighborhood of any fixed point $p\in M^2$. Therefore, $M^2$ has constant principal curvatures (and, then constant mean curvature). Moreover, recall that $M^2$ is also totally biharmonic. That is, the shape operator of $M^2$, $S_\eta$ with $\eta=\hat{E}_3$, must verify $\lvert S_\eta\rvert^2=\Ricci\left(\eta,\eta\right)$. Thus, together with equation (16) of \cite{GM}, we obtain
\begin{equation}
\lvert S_{\eta} \rvert^2=\kappa_g^2+\frac{b^2}{2}=\Ricci (\hat{E}_3,\hat{E}_3)=4\,a-\frac{b^2}{2}\, ,\nonumber
\end{equation}
where $\kappa_g$ denotes the geodesic curvature of the curve $\alpha(s)=\pi(\widetilde{\alpha}(s))$ in the base space of the Riemannian submersion (see \cite{GM} or \cite{M} for details). Moreover, $\kappa_g^2=4\,a-b^2> 0$, thus, the curve $\alpha(s)$ has constant curvature in a surface with constant Gaussian curvature $4\,a>b^2\geq 0$, that is, in a 2-dimensional round sphere $\mathbb{S}^2(4\,a)$.\\
In conclusion we have that the curve $\alpha(s)$ represents a circle and, as a consequence, the Hopf cylinder $M^2$  is invariant under the action of the group $SO(2)$, that is, $M^2$ is a rotational Hopf cylinder.
\end{proof}

As a consequence  of Theorem~\ref{hopf}, in order to classify the totally biharmonic surfaces in BCV spaces, we just need to find the rotational Hopf cylinders such that all their geodesics are biharmonic in the BCV ambient space. Thus, from now on, we are going to assume that $M^2$ is a $SO(2)$-invariant Hopf cylinder of a BCV space $N(a,b)$ with $4\,a\neq b^2$. Using the adapted orthonormal frame introduced in \eqref{eq:ecap}, $\{\hat{E}_1,\hat{E}_2,\hat{E}_3\}$, we have that the unit tangent vector field of any arc-length parametrized curve immersed in $M^2$, $\gamma(s)$, can be written as
\begin{equation}
\mathbf{t}(s)=\gamma'(s)=\sin\omega\, \hat{E}_1+\cos\omega\,\hat{E}_2\, , \nonumber
\end{equation}
for some function $\omega$. That is, the tangent vector field of any curve $\gamma$ of $M^2$, seen as a curve in the ambient  space $N(a,b)$, can be described in terms of the orthonormal frame $\{E_1,E_2,E_3\}$ as follows
\begin{equation}
\mathbf{t}=\sin\omega\,\cos\vartheta\,E_1+\sin\omega\,\sin\vartheta\,E_2+\cos\omega\,E_3\, , \label{tangent}
\end{equation}
where  $\vartheta$ stays constant along the integral curves of $E_3$.

Moreover, if we also assume that $\gamma$ is a geodesic in $M^2$ we have that $\omega$ must be constant. In fact, from \eqref{tangent} and taking into account \eqref{eq:levi-civita-BCV}, we can compute $\nabla_\mathbf{t}E_3$, obtaining 
\begin{equation}
\langle \mathbf{t},\nabla_\mathbf{t} E_3\rangle=\langle \mathbf{t},\sin\omega\,\cos\vartheta\,\nabla_{E_1}E_3+\sin\omega\,\sin\vartheta\,\nabla_{E_2}E_3\rangle=\frac{b}{2}\sin\omega\,\langle \mathbf{t}, \hat{E}_3\rangle=0\, . \nonumber
\end{equation} 
Thus, we have
\begin{equation}
\frac{d}{ds}\left(\cos\omega\right)=\mathbf{t}\langle\mathbf{t},E_3\rangle=\langle\nabla_\mathbf{t}\,\mathbf{t},E_3\rangle+\langle\mathbf{t},\nabla_\mathbf{t} E_3\rangle=\langle II(\mathbf{t},\mathbf{t})\,\hat{E}_3, E_3\rangle=0\, ,\nonumber
\end{equation}
where we have used the Gauss formula and that  $\gamma$ is a geodesic in $M^2$ to get $\nabla_\mathbf{t}\,\mathbf{t}=II(\mathbf{t},\mathbf{t})\,\hat{E}_3$. That is, $\cos\omega$ is constant along $\gamma$, which means that $\omega$ is a constant function. In other words, geodesics of a rotational Hopf cylinder are helices making a constant angle with the vertical direction $\hat{E}_2=E_3$.

Observe that these helices in BCV spaces $N(a,b)$ can be parametrized by  arc-length as follows. For any real constants $r$ and $\mu$, take the curve
\begin{equation}
\gamma(s)=\left(r\sin\lambda\, s,-r\cos\lambda\,s, \lambda\,\mu\,s  \right)\, .\label{par}
\end{equation}
Now, the arc-length condition implies that the function $\lambda$ must satisfies 
\begin{equation}
\lambda=\frac{\lambda_a}{\sqrt{r^2+\left(\frac{b}{2}\,r^2-\mu\,\lambda_a\right)^2\,}}\, .\label{lambdaa}
\end{equation} 
We are going to call the constant $r$ the \emph{Euclidean radius} of the helix $\gamma$ by  analogy with the Euclidean space $\mathbb{R}^3$.

It turns out that,  combining the derivative of the parametrization \eqref{par} of any of these helices with the expression of the tangent vector field given in \eqref{tangent} and some manipulations, the following relations between the constant function $\omega$ and the parameters $r$ and $\mu$ defining the helix $\gamma$ are valid
\begin{equation}
\sin\omega=\frac{r}{\sqrt{r^2+\left(\frac{b}{2}\,r^2-\mu\,\lambda_a\right)^2\,}}\,, \quad\quad \cos\omega=\frac{\frac{b}{2}\,r^2-\mu\,\lambda_a}{\sqrt{r^2+\left(\frac{b}{2}\,r^2-\mu\,\lambda_a\right)^2\,}}\, . \label{sinomega}
\end{equation}

We thus obtain the following result

\begin{prop}\label{Hr} A rotational Hopf cylinder into $N(a,b)$, $\rho=4\,a\neq b^2$, is a totally biharmonic surface if and only if $N(a,b)\cong \mathbb{S}^2(\rho)\times \mathbb{R}$ and the Euclidean radius of all the helices $\gamma$ of the Hopf cylinder is given by $$r^2=\frac{3\pm2\sqrt{2}}{a}.$$
\end{prop}
\begin{proof} Let $M^2$ be a rotational Hopf cylinder of $N(a,b)$ with $4\,a\neq b^2$. As mentioned before, geodesics of $M^2$ seen as curves in $N(a,b)$ are the helices $\gamma$ parametrized by \eqref{par} and whose tangent vector field $\mathbf{t}$ is given in \eqref{tangent}. Both formulas are related due to the equations \eqref{lambdaa} and \eqref{sinomega}.\\
First of all, if $\gamma$ happens to be also a geodesic in the ambient space $N(a,b)$ we are done. This is the case, for instance, for the integral curves of $E_3$. Lets us assume that $\gamma$ is not a geodesic of the ambient space and, therefore, along $\gamma$ there is a well-defined Frenet frame $\{\mathbf{t},\mathbf{n},\mathbf{b}\}$. Then,  computing the covariant derivative of the tangent vector field $\mathbf{t}$, \eqref{tangent}, we  obtain the following explicit expression of the curvature of $\gamma$, $\kappa$, as a curve in $N(a,b)$ (see \cite[Lemma 5.5]{CMOP} for the explicit computation):
\begin{equation}
\kappa=\varpi\,\sin\omega\, ,\label{kappa}
\end{equation}
where $\varpi$ is a constant given by
\begin{equation}
\varpi=\lambda-2\,a\,r\,\sin\omega-b\,\cos\omega  \, ,\nonumber
\end{equation} 
while  $\lambda$, $\sin\omega$ and $\cos\omega$ are given in \eqref{lambdaa} and \eqref{sinomega}, respectively.\\
Moreover, from this computation and the constant value of $\omega$ proved before, we easily obtain that $\mathbf{n}_3=\langle \mathbf{n},E_3\rangle=0$, $\mathbf{n}$ being the Frenet normal to $\gamma$. In particular, this means that $\mathbf{b}_3=\langle \mathbf{b},E_3\rangle=\sin\omega$, since the Frenet binormal vector field can be defined as $\mathbf{b}=\mathbf{t}\times\mathbf{n}$, where here $\times$ denotes the cross product.\\
Then, arguing as in Theorem 5.6 of \cite{CMOP} we can also compute the torsion $\tau$ of $\gamma$ obtaining that
\begin{equation}
\tau=-\varpi\,\cos\omega-\frac{b}{2}\, . \label{tau}
\end{equation}
Now, the curve $\gamma$ in $N(a,b)$, $4\,a\neq b^2$, is a proper biharmonic curve if equations \eqref{b11}-\eqref{b14} are verified. Equation \eqref{b11} trivially holds because both $\varpi$ and $\omega$ are constants and $\gamma$ is non-geodesic in $N(a,b)$. Moreover, the same argument works to prove that equation \eqref{b12} is verified. Equation \eqref{b13} is also true since, as argued before, $\mathbf{n}_3=0$. Thus, in order to check if $\gamma$ is a proper biharmonic curve in $N(a,b)$, we just need to verify equation \eqref{b14}.\\
Combining \eqref{b14} with the values of the curvature of $\gamma$, \eqref{kappa}, the torsion, \eqref{tau}, and with $\mathbf{b_3}=\sin\omega$ we get, after some simplifications,
\begin{equation}\label{eq:rr}
a\left(2a\left[1-\mu b\right]+b^2\right)r^4+\left(b^2-12a\right)r^2+2\left(1+\mu b\right)=0\,. 
\end{equation}
Notice that, in order to have that all geodesics $\gamma$, \eqref{par}, of $M^2$ are proper biharmonic helices in the ambient BCV space, $N(a,b)$, we need that the radius $r$, which is a solution of \eqref{eq:rr},  must be independent of the parameter $\mu$. This only happens when $b=0$, that is, when $N(a,b)=N(a,0)$. Moreover, for $b=0$ there are exactly two positive solutions of \eqref{eq:rr}, given by
\begin{equation}
r^2=\frac{3\pm2\sqrt{2}}{a}\, \label{r}
\end{equation}
provided that $a>0$. In other words, the rotational Hopf cylinder $M^2$ is totally biharmonic if and only if  $N(a,b)$ is congruent to $\mathbb{S}^2(4\,a)\times\mathbb{R}$ and $r$ is one of the two possible positive radius given in \eqref{r}.
\end{proof}

\begin{rem}\label{rm-last}
In principle,  Proposition~\ref{Hr} shows the existence of two non-totally geodesic surfaces $M^2$ which are totally biharmonic in the BCV space $N(a,0)\cong \mathbb{S}^2(\rho)\times\mathbb{R}$ where $\rho=4\,a$. These surfaces are rotational Hopf cylinders and they are completely described by the positive radius, $r>0$, given in \eqref{r}. Therefore, for notation consistence we are going to denote by $M^2_+$ (and, $M^2_-$) the associated surface with $r_+$ (and $r_-$, respectively) determined by the choice of sign in \eqref{r}.

Now, it turns out that both surfaces are essentially the same, in the sense that they are congruent to each other. Moreover, they are exactly the same up to a change of orientation. In order to prove this, consider the parametrization $\mathbf{x}(s,t)=\psi_t(\widetilde{\alpha}(s))$ of the Hopf cylinder where $\{\psi_t\,;\,t\in\mathbb{R}\}$ denotes the one-parameter group of isometries associated to $E_3$ and $\widetilde{\alpha}(s)$ is an arc-length parametrized curve in $N(a,b)$ everywhere orthogonal to $E_3$. We can easily compute the first fundamental form of $M^2_+$ and $M^2_-$, verifying that both of them are the same.
Moreover, a straightforward computation of the coefficients of the second fundamental form gives
\begin{equation}
II_{11}=\frac{1-a\,r^2_{\pm}}{r_{\pm}}\, , \quad\quad II_{12}=II_{22}=0\, . \nonumber
\end{equation}
Thus, after substituting the values of $r_{\pm}$ given in \eqref{r}, we have that $II_{11}^+=-2\sqrt{a}=-II_{11}^-$, which means that both surfaces, $M_+^2$ and $M^2_-$, are the same up to a change of orientation.
\end{rem}

Thus combining Proposition~\ref{Hr} and Remark~\ref{rm-last} we have  the following result

\begin{thm}\label{clasBCV} A totally biharmonic surface into $N(a,b)$, with $\rho=4\,a\neq b^2$, is either a part of a totally geodesic surface or a part of the cylinder $\mathbb{S}^1(\sqrt{2\,\rho}\,)\times\mathbb{R}$ where $N(a,b)\cong\mathbb{S}^2\left(\rho\right)\times\mathbb{R}$.
\end{thm}

Finally, combining the case $n=3$ of Theorem~\ref{clas} and Theorem~\ref{clasBCV}, we obtain the local classification of totally biharmonic surfaces into BCV spaces given, as announced, in Theorem~\ref{relconj}.


\begin{thebibliography}{777}


\bibitem{new} A. Balmu\c s, S. Montaldo and C. Oniciuc. Classification results for biharmonic submanifolds in spheres. \textit{Isr. J. Math.} \textbf{168} (2008), 201--220.

\bibitem{BMO13} A. Balmu\c s, S. Montaldo and C. Oniciuc. Biharmonic PNMC submanifolds in spheres. \textit{Ark. Mat.} \textbf{51} (2013), 197--221.


\bibitem{nuevo} R. Caddeo, S. Montaldo and C. Oniciuc, Biharmonic submanifolds in spheres. \textit{Isr. J. Math.} \textbf{130} (2002), 109--123.

\bibitem{CMOP} R. Caddeo, S. Montaldo, C. Oniciuc and P. Piu. The Euler-Lagrange method for biharmonic curves. \textit{Mediterr. J. Math.} \textbf{3} (2006), 449--465.

\bibitem{Ca1} E. Cartan. Famillies de surfaces isoparametriques dans les espaces a courbure constante. \textit{Annali di Mat.} \textbf{17} (1938), 177--191.

\bibitem{Ca} E. Cartan. \textit{Lecons sur la Geometrie des Espaces de Riemann}. Gauthier Villars, Paris, 1946.

\bibitem{C84} B-Y. Chen. {\em Total Mean Curvature and Submanifolds of Finite Type}. Series in Pure Mathematics~1. World Scientific Publishing Co., Singapore, 1984.

\bibitem{DVM} M. Dom\'inguez-V\'azquez and J. M. Manzano. Isoparametric surfaces in $\mathbb{E}(\kappa,\tau)$-spaces. arXiv:1803.06154 [math.DG], (2018).

\bibitem{Eells-Sampson} J. Eells and J. H. Sampson. Harmonic mappings of Riemannian manifolds. \textit{Amer. J. Math.} \textbf{86} (1964), 109--160.

\bibitem{ES} J.~Eells and J. H.~Sampson. Variational theory in fibre bundles. {\em Proc. U.S.-Japan Seminar in Differential Geometry}, Kyoto (1965), 22--33.

\bibitem{GM} O. J. Garay and S. Montaldo. Elasticae in Killing submersions. \textit{Mediterr. J. Math.} \textbf{13} (2016), 3281--3302.

\bibitem{urakawa} T. Ichiyama, J. Inoguchi and H. Urakawa. Classifications and isolation phenomena of bi-harmonic maps and bi-Yang-Mills fields. \textit{Note Mat.} \textbf{30} (2010), 15--48.

\bibitem{IM} D. Impera and S. Montaldo. Totally Biharmonic Submanifolds. Differential geometry, 237--246, World Sci. Publ., Hackensack, NJ, 2009. 

\bibitem{J} G. Y. Jiang. $2$-Harmonic maps and their first and second variational formulas. \textit{Chinese Ann. Math. Ser. A} \textbf{7} (1986), 389--402. Translated from the Chinese by Hajime Urakawa. Note Mat. 28 (2009), 209--232.


\bibitem{LM08} E.~Loubeau and S.~Montaldo. Biminimal immersions. \textit{Proc. Edinb. Math. Soc.} \textbf{51} (2008), 421--437.

\bibitem{M}  J. M. Manzano. On the classification of Killing submersions and their isometries. \textit{Pacific J. Math.} \textbf{270} (2014), 367--392.

\bibitem{O10} Y. -L.~Ou. Biharmonic hypersurfaces in Riemannian manifolds. {\em Pacific J. Math.} \textbf{248} (2010), 217--232.

\bibitem{Vr}  G. Vranceanu. \textit{Lecons de Geometrie Differentielle}. Ed. Acad. Rep. Pop. Roum., vol I, Bucarest, 1957.

\end{thebibliography}
\end{document}